\documentclass[10 pt, leqno]{amsart}

\usepackage {amsfonts}
\usepackage{amsthm}
\usepackage{amssymb}
\usepackage{latexsym}
\usepackage{amsmath}
\usepackage{mathrsfs}

\theoremstyle{plain}
\newtheorem{tw}{Theorem}[section]

\newtheorem {lem} [tw]{Lemma}

\theoremstyle{definition}

\newtheorem {rem} [tw]{Remark}

\newcommand{\bc} {\Bbb C}
\newcommand{\bn}{\Bbb N}

\newcommand{\be}{\Bbb E}

\newcommand{\ld}{\ldots}

\newcommand{\alg} {\mathsf{A}}
\newcommand{\blg} {\mathsf{B}}

\newcommand {\hte} {{\textrm{ht}}}
\newcommand {\Lin} {{\textup{Lin}}}

\newcommand{\edom}{\rho_{\sigma}}
\newcommand{\uend}{\rho_u}

\newcommand{\Ind}{\mathcal{J}}

\newcommand{\Pp}{\mathcal{P}}

\newcommand{\can}{\Psi}

\newcommand{\Cun}{\mathcal{O}_N}
\newcommand{\cun}{\mathcal{O}_N}
\newcommand{\twoUHF}{\mathcal{F}_2}
\newcommand{\twoCun}{\mathcal{O}_2}
\newcommand{\twomasa}{\mathcal{C}_2}
\newcommand{\subs}{\subseteq}

\newcommand{\nUHF}{\mathcal{F}_N}
\newcommand{\nmasa}{\mathcal{C}_N}

\newcommand{\ot}{\otimes}

\newcommand{\wt}{\widetilde}

\numberwithin{equation}{section}

\begin{document}

\author{Adam Skalski}
\address{Department of Mathematics and Statistics,  Lancaster University,
Lancaster, LA1 4YF} \email{a.skalski@lancaster.ac.uk}
\author{Joachim Zacharias}
\footnote{\emph{Permanent address of the first named author:} Department of Mathematics, University of \L\'{o}d\'{z}, ul. Banacha
 22, 90-238 \L\'{o}d\'{z}, Poland.}
\address{School of Mathematical Sciences,  University of Nottingham,
Nottingham, NG7 2RD}

 \email{joachim.zacharias@nottingham.ac.uk  }

\title{\bf Noncommutative topological entropy of endomorphisms of Cuntz algebras}


\begin{abstract}
\noindent Noncommutative topological entropy estimates are obtained for polynomial gauge invariant
endomorphisms of Cuntz algebras, generalising known results for the canonical shift endomorphisms.
Exact values for the entropy are computed for a
class of permutative endomorphisms related to branching function systems introduced and studied
by Bratteli, Jorgensen and Kawamura.
\end{abstract}

\maketitle

In \cite{Voic} D.\,Voiculescu defined noncommutative topological entropy for automorphisms (or completely positive maps) of
nuclear $C^*$-algebras based on completely positive approximations which naturally extends the classical definition of
topological entropy for continuous maps of compact spaces. His definition has been subsequently extended to the larger class of
exact $C^*$-algebras (by Brown \cite{NateBrown}) and intensively studied in the past decade (we refer to the book \cite{book} for
a comprehensive discussion and many examples). For Cuntz algebras and its various generalisations entropy estimates have only
been obtained for canonical endomorphisms which may be regarded as noncommutative extension of classical shift maps. In
\cite{Choda} M.\,Choda showed that the noncommutative topological entropy of the canonical shift endomorphism of the Cuntz
algebra $\Cun$ is equal to $\log N$. Later in \cite{Boca} F.\,Boca and P.\,Goldstein used a different method which allowed them
to compute entropy for the shift-type endomorphisms on arbitrary Cuntz-Krieger algebras. Their techniques were extended in
\cite{Hous} to determine the values of noncommutative entropy and pressure for the multidimensional shifts on $C^*$-algebras
associated with higher-rank graphs. In all of these results it turned out that the entropy of the canonical shift endomorphism is
the same as the entropy of the corresponding classical shift.

In this paper we try to estimate entropies for more general endomorphisms of Cuntz algebras.
Endomorphisms of Cuntz algebras have been studied intensivly and have applications in subfactors,
quantum field theory and other areas. A particularly interesting class is formed by those which
leave $\nUHF$, the canonical UHF-subalgebra of $\Cun$, invariant. Besides the canonical shift
this class contains the recently introduced and intensively studied  \emph{permutative polynomial endomorphisms}.
We refer to \cite{Brat} and \cite{Kawa} for connections with branching function systems and permutative representations
and [CS$_{1-2}$] and \cite{cks} for connections with subfactors and Mathematical Physics. In particular \cite{cs2}
draws an interesting connection between our entropy estimates and indices of endomorphisms on $\twoCun$.

In the first part of this paper we give a general upper bound (in Theorem \ref{upest}) for the
entropy of endomorphism which leave $\nUHF$ invariant and which verify a certain `finite-range' condition
(polynomial endomorphisms).
Easy examples show that this bound is sharp. We note that the same methods apply to similar endomorphisms of
Cuntz-Krieger algebras, graph and even higher-rank graph $C^*$-algebras.

In the second part we obtain exact values of the entropy for all polynomial endomorphisms of
$\twoCun$ coming from permutations of rank $2$ (fully classified in \cite{Kawa}).
The results of the entropy computations are given in
the Table 1 in Section 3.
Just as the canonical endomorphism
they all leave the standard maximal abelian subalgebra $\twomasa$ of $\twoCun$ invariant and thus correspond there to
a transformations of its spectrum, the Cantor set. But contrary to the
case of the shift endomorphism it may happen that the entropy of the endomorphism is greater than
the entropy of its restriction to $\twomasa$.
However, for all permutative endomorphisms we consider we can find other non-standard masas
which are invariant and such that the restriction to this masa has the same entropy as the
whole endomorphism.

Our results lead to the following questions about endomorphisms of Cuntz algebras which we have not been able to answer.
\begin{enumerate}
\item Are there polynomial (or more general) endomorphisms for which the noncommutative topological
entropy is strictly greater than the supremum of the entropies over all classical (commutative) subsystems?
In general there are only few known examples of $C^*$-dynamical systems for which this may happen.
(E.g.\;shifts on $C^*$-algebras of certain bitstreams have this property; c.f.\;Chapter 12 in \cite{book}.)
\item Is the entropy of a polynomial (or more general gauge invariant) endomorphism  always equal
to the entropy of its restriction to the UHF-algebra $\nUHF$?
\end{enumerate}

\section{Cuntz algebras and their endomorphisms}

\subsection*{Definition and basics of Cuntz algebras} Let $N\in \bn$. Recall (\cite{Cuntz1}) that the \emph{Cuntz algebra} $\cun$ is the unique $C^*$-algebra generated
by isometries $s_1,\ld,s_N$ subject to the relations
$$
s_i^* s_j =\delta_{i,j}1 \text{ and } \sum_{i=1}^N s_is_i^*=1.
$$
$\cun$ is simple and nuclear. To describe elements in $\cun$ it is convenient to introduce the following multi-index notation.
For $k \in \bn$ the set of multi-indices of length $k$ with values in $\{1, \ldots, N\}$ will be denoted by $\Ind_k$, so that
$\Ind_k=\{1, \ldots, N\}^k$; we write $\Ind_0 = \{\emptyset\}$ and $\Ind = \cup_{k\in \bn_0} \Ind_k$ (where $\bn_0=\bn \cup
\{0\}$). Multi-indices in $\Ind$ will be denoted either by capital Latin letters $I,J,\ldots$ (as in \cite{Brat}) or little Greek
letters $\mu, \nu, \ldots$ (as in \cite{Cuntz1}), with the standard notation for concatenation: if $I=(i_1,\ldots,i_k) \in
\Ind_k$, $J=(j_1,\ldots,j_l) \in \Ind_l$, then $IJ:=(i_1,\ldots,i_k,j_1,\ldots,j_l) \in \Ind_{k+l}$. The length of $I \in \Ind_k$
is the number $k$ denoted by $|I|=k$. We will often write $p=|P|$, $l=|L|$, etc.

For $J \in \Ind_k$, $J=(j_1,\ldots,j_k)$ we write $s_J = s_{j_1} \cdots s_{j_{k}}$ (and
$s_{\emptyset}:=1$). It follows from the relations defining $\cun$ that every element
in the $*$-algebra generated by $s_1,\ld,s_N$ can be written as a linear combination
of expressions of the form $s_I s_J^*$, where $I, J \in \Ind$.
This dense $*$-subalgebra of $\Cun$ is the \emph{polynomial subalgebra} $\Pp(\Cun)$ of $\Cun$.

Notice that the span of $\{s_I s_J^*: I, J \in \Ind_k\}$ is isomorphic to $M_{N^k}$ (also denoted $F_k$), where  $s_I s_J^*$
corresponds to the matrix unit $e_{I,J}$. The subalgebra generated by $\bigcup_k \{s_I s_J^*: I, J \in \Ind_k\}$ is a UHF-algebra
of constant multiplicity $N$ which will be denoted by $\nUHF$. There is a conditional expectation $\be : \cun \to \nUHF$ which is
given by integration over the gauge action (c.f.\;\cite{Cuntz1}). The \emph{standard masa} (maximal abelian subalgebra) is the
algebra $\nmasa$ generated by $\{s_I s_I^*: I \in \Ind\}$.  It is maximal abelian in $\cun$ and $\nUHF$. So we have a chain of
natural inclusions
$$
\nmasa \subs \nUHF \subs \cun
$$
Besides the standard masa  $\nmasa$ we will consider some other masas in $\nUHF$. Recall that a
masa $\mathcal{D}$ in $\nUHF$ is said to be a \emph{product masa} if it is of the form
$\bigotimes_{i=1}^{\infty} C_i$, where each $C_i$ is a masa in $M_N$. Any two product masas are
approximately unitarily equivalent but will be conjugate only in exceptional cases. \emph{Special
product masas} are of the type $\mathcal{C} = \bigotimes_{i=1}^{\infty} C_i$, where $C_i=C_j$
for all $i, j \in \bn$. It is not hard to see that each such $\mathcal{C}$ is isomorphic to
the algebra of continuous functions on the infinite product $\prod_{i=1}^{\infty} \{1 ,\ld , N\}$
of $N$ letters, which is in turn is homeomorphic to the Cantor set $\mathfrak{C}$.

\subsection*{Endomorphisms of $\cun$}
The \emph{canonical shift endomorphism} $\theta:\Cun \to \Cun$ is given by the formula
\[
\theta(x) = \sum_{i=1}^{N} s_i x s_i^*, \;\;\; x\in \Cun.
\]
It leaves $\nUHF$ and $\nmasa$ invariant, and also every special product masa. It is easy
to see that if $\mathcal{C}$ is a special product masa and if we regard $ x \in \mathcal{C}$
as a function on $ \mathfrak{C}$ then $\theta (x) = x \circ T$, where
$T: \mathfrak{C} \to \mathfrak{C}$
is the usual one-sided shift on $\mathfrak{C}=\prod_{i=1}^{\infty} \{1 ,\ld , N\}$.

It is well-known (\cite{Cuntz}) that there is a bijective correspondence between
$*$-endomorphisms of the Cuntz algebra $\Cun$ and unitaries in $\Cun$. Given a unitary
$u\in \Cun$, $us_1 , \ldots ,us_N$ verify again the Cuntz algebra relations so that
$s_i \mapsto us_i$ for $i=1,\ldots,N$ extends uniquely to a unital endomorphism $\uend$.
Conversely, given a unital endomorphism $\rho$, $u_{\rho}=\sum_{i=1}^N\rho (s_i)s_i^*$
is a unitary such that $\rho_{u_\rho}=\rho$ and $u_{\rho_u}=u$. Easy examples are the
\emph{gauge automorphism}, where $u=\lambda 1$, with $\lambda$ a complex number of
modulus 1, or the canonical shift endomorphism, where
$u=\sum_{i,i=1}^N s_is_j s_i^* s_j^* \in \nUHF$.

A description of the action of $ \uend$ on higher monomials can be given as follows: for $k \in \bn$ if we define
\begin{equation}\label{cocyc}
u_k=u\theta (u) \ldots \theta^{k-1}(u),
\end{equation}
where $\theta$ is the shift endomorphism defined above,
then $ \uend ( s_I)=u_k s_I$ for all multi-indices $I \in \Ind_k$ and
thus $ \uend ( s_I s_J^*)=u_k s_I s_J^* u_l^*$ for $I \in \Ind_k$ and
$J \in \Ind_l$.

If $x \in \nUHF$ then $ \uend (x)= \lim_{k \to \infty} u_kxu_k^*$;
it is also known that $ \uend $ preserves $\nUHF$ iff $u \in \nUHF$
iff $ \uend $ commutes with all gauge automorphisms (\cite{Cuntz}).
These endomorphisms are thus called \emph{gauge invariant} and
for them we also have $ \tau \circ \be \circ \uend = \tau \circ \be$
where $\tau$ the unique trace on $\nUHF$.

\emph{Permutative endomorphisms} of $\Cun$ are defined as follows. Suppose that $k \in \bn$ and $\sigma$ is a permutation of the
set $\Ind_k$. Put
\[u_{\sigma} = \sum_{J \in \Ind_k} s_{\sigma(J)} s_J^*.\]
It is easy to check that $u_{\sigma}$ is a unitary in $\Cun$ more precisely it lies in $F_k=M_{N^k}$ and can be regarded as the
permutation matrix corresponding to $\sigma$. The permutative endomorphism $\edom$ corresponding to the permutation $\sigma$ is
defined as $\rho_{u_{\sigma}}$. For instance the canonical endomorphism $\theta$ is permutative. Several other examples will be
discussed in detail in the last section of this note. Remark here only that the identity, canonical shift endomorphism and the
flip automorphisms exchanging the generators all belong to this class.

Finally the following is evident from the above correspondence between unitaries and
endomorphisms.
\begin{rem}
$ \uend$ leaves the polynomial subalgebra $\Pp(\Cun)$ invariant iff $u \in \Pp(\Cun)$.
\end{rem}
\noindent
Such endomorphisms are thus called \emph{polynomial endomorphisms} and we will only consider
those in this paper.

\subsection*{Three technical Lemmas}
Define for $k, l \in \bn_0$
\begin{equation}
\label{not1} A_{k,l} = \{s_I s_J^*: I \in \Ind_k, J \in \Ind_l\}.
\end{equation}
and
\begin{equation}
\label{not2} F_{k,l} = \Lin \; A_{k,l}.
\end{equation}
$\bigcup_{k,l \in \bn_0}A_{k,l}$ is total in $\Cun$ and the linear span
$\Lin \left(\bigcup_{k,l \in \bn_0}A_{k,l}\right)= \bigcup_{k,l\in \bn_0} F_{k,l} $
is just the polynomial subalgebra $\Pp(\cun)$.

There is a well-known $*$-isomorphism between $\Cun$ and $M_{N^k} \ot \Cun$ defined by
\[ \can_k (X) = \sum_{K,M \in \Ind_k} e_{K,M} \ot s_K^* X s_M, \;\;\; X \in \Cun,\]
where $e_{K,M}$ denote the matrix units in $M_{N^k}$ as before.

The following Lemma is elementary, but crucial for what follows. It shows that neither in
Lemma 2 of \cite{Boca} nor in Lemma 2.2 of \cite{Hous} was it essential that one dealt with the
shift-type transformations.

\begin{lem} \label{new}
Let $k,p,l \in \bn$, $k \geq \max\{p,l\}$. Suppose that $X \in F_{p,l}$. If $p > l$ then
\begin{equation} \can_k(X) =
\sum_{J \in \Ind_{p-l}} T_J \ot s_J,\label{form} \end{equation}
where   $T_J \in M_{N^k}$, $\|T_J\|\leq \|X\|$ for each $J \in \Ind_{p-l}$; if $p<l$ then
\[ \can_k(X) = \sum_{J \in \Ind_{l-p}} T_J \ot s_J^*,\]
where $T_J \in M_{N^k}$, $\|T_J\|\leq \|X\|$ for each $J \in \Ind_{l-p}$ . Finally if $p=l$ then
\[\can_k(X) = T \ot 1,\]
where $T \in M_{N^k}$, $\|T\|\leq \|X\|$.
\end{lem}
\begin{proof}
Suppose that $p>l$ and let $X= \sum_{P\in \Ind_p, L \in \Ind_l} \gamma_{P,L} s_P s_L^*,$ where $\gamma_{P,L} \in \bc$. Then
\begin{align*}  \can_k (X) &= \sum_{K,M \in \Ind_k}   \sum_{P\in \Ind_p, L \in \Ind_l} \gamma_{P,L} e_{K,M} \ot s_K^* s_P s_L^* s_M \\
&=
 \sum_{P\in \Ind_p, L \in \Ind_l} \sum_{K'\in \Ind_{k-p}, M' \in \Ind_{k-l}}  \gamma_{P,L} e_{PK',LM'} \ot s_{K'}^* s_{M'} \\&=
 \sum_{P\in \Ind_p, L \in \Ind_l} \sum_{K'\in \Ind_{k-p}, J \in \Ind_{p-l}}  \gamma_{P,L} e_{PK',LK'L'} \ot  s_{J}.\end{align*}
This shows that \eqref{form} holds for some $T_J \in M_{N^k}$.

Observe now that for each $n,k \in \bn$ and a family of complex $n\times n$ matrices $\{\alpha_J:J \in \Ind_m\}$ we have
\begin{align*}
 \|\sum_{J \in \Ind_{m}} \alpha_J \ot s_J \|^2 &= \| ( \sum_{J \in \Ind_{m}} \alpha_J \ot s_J)^* (\sum_{K \in \Ind_{m}}
\alpha_K \ot s_K)\| \\& = \|  \sum_{J,K \in \Ind_{m}} \alpha_J^* \alpha_K  \ot s_J^*s_K \| = \|  \sum_{J \in \Ind_{m}} \alpha_J^*
\alpha_J  \ot 1 \| =  \|  \sum_{J \in \Ind_{m}} \alpha_J^* \alpha_J  \|,
\end{align*}
so in particular for any fixed $K \in\Ind_m$ we have \[\|\alpha_K\| \leq \|\sum_{J \in \Ind_{m}} \alpha_J \ot s_J \|.\] Now
connecting the above with the fact that $\can_k$ is a $^*$-homomorphism, we get for each $J \in \Ind_m$
\[ \|T_J\| \leq \|\can_k (X)\| \leq \|X\|\]
and the proof is finished. The cases $p<l$ and $p=l$ follow in a similar way.
\end{proof}

Analogous statements remain true in the context of the Cuntz-Krieger, graph and higher-rank graph $C^*$-algebras. It can be used
to estimate the entropy of a completely positive contractive map on the $C^*$-algebra of one of the types listed above if only we
have control on `how far' the map sends the canonical matrix units.

\begin{lem} \label{new2}
Let $k \in \bn$
and let $u\in F_k$ be a unitary. Then for all $m,p,l \in \bn$
\[\uend^m (F_{p,l}) \subseteq F_{p+m(k-1),
l+m(k-1)}.\]
\end{lem}
\begin{proof}
It suffices to show this for $m=1$.
Using the formula $\uend (s_P s_L^*)=u_ps_P s_L^*u_l^*$, where $p=|P|$ and $l=|L|$ (and $u_p$ and $u_l$ are as in \;(\ref{cocyc})), this follows
since if $u\in F_k$ then $u_p \in F_{p+k-1}$ and $ u_l \in F_{l+k-1}$.
\end{proof}
\noindent
If $u$ is a permutation matrix we have a stronger statement.
\begin{lem} \label{inv}
Let $\sigma$ be a permutation of $\Ind_k$, $p,l,m \in \bn$. Then
\begin{equation} \edom^m (A_{p,l}) \subseteq A_{p+m(k-1),
l+m(k-1)}. \label{rough}\end{equation}
Further $\edom$ leaves both $\nUHF$ and $\nmasa$ invariant.
\end{lem}
\begin{proof}
Similar to Lemma \ref{new2}.
\end{proof}
\noindent
If a unitary $u \in F_{k}$ is not a permutation matrix, the endomorphism $\rho_u$ need not leave $\nmasa$ invariant.

\section{An Entropy Bound for Gauge Invariant Polynomial Endomorphisms}

\subsection*{Topological Entropy}
Let $\alg$ be a unital $C^*$-algebra. We say that $(\phi,\psi,C)$ is an \emph{approximating triple}
for $\alg$ if $C$ is a finite-dimensional $C^*$-algebra and both $\phi:C\to \alg$, $\psi:\alg \to C$
are unital and completely positive (ucp). This will be indicated by writing $(\phi,\psi,C)\in CPA(\alg)$.
Whenever $\Omega$ is a finite subset of $\alg$ ($\Omega \in FS(\alg)$) and
$\varepsilon >0$ the statement
$(\phi,\psi,C)\in CPA(\alg, \Omega, \varepsilon)$ means that $(\phi,\psi,C)\in CPA(\alg)$ and
for all $a \in \Omega$
\[
\|\phi \circ \psi(a) - a \| < \varepsilon.
\]
Nuclearity of $\alg$ is equivalent to the fact that for each $\Omega \in FS(\alg)$ and $\varepsilon >0$ there exists a triple
$(\phi,\psi,C)\in CPA(\alg, \Omega, \varepsilon)$. For such algebras  one can define
\[
\textrm{rcp}(\Omega, \varepsilon)= \min\{\textrm{rank}\,C:\; (\phi,\psi,C)\in CPA(\alg, \Omega, \varepsilon)\},
\]
where $\textrm{rank}\,C$ denotes the dimension of a maximal abelian subalgebra of $C$. Let us recall the definition of
noncommutative topological entropy in nuclear unital $C^*$-algebras, due to D.Voiculescu (\cite{Voic}). Assume that $\alg$ is
nuclear and $\gamma:\alg\to \alg$ is a ucp map. For any $\Omega \in FS(\alg)$ and $n\in \bn$ let
\begin{equation}
\label{enning} \text{orb}^n(\Omega)=\Omega^{(n)} = \bigcup_{j=0}^n {\gamma^j(\Omega)}.
\end{equation}
Then the (Voiculescu) \emph{noncommutative topological entropy} is given by the formula:
\[
\hte(\gamma) = \sup_{\varepsilon>0, \, \Omega \in FS(\alg)} \limsup_{n\to\infty} \left(\frac{1}{n} \log
\textrm{rcp}(\Omega^{(n)}, \varepsilon)\right).
\]
As shown in \cite{Voic} Proposition 4.8 the approximation entropy coincides with classical
topological entropy in the commutative case (see \cite{Walt}, Chapter 7 for the definition of the latter).
Another important property to be used in the sequel is the fact that the entropy decreases when
passing to invariant subalgebras. More precisely,
\begin{rem}
Let $\alg$ be a nuclear $C^*$-algebra, $\gamma : \alg \to \alg$ ucp and $\blg \subset \alg$ a nuclear subalgebra such that
$\gamma (\blg) \subseteq \blg$. Then $\hte(\gamma|_{\blg})\leq \hte (\gamma)$, where $\hte$ denotes the Voiculescu entropy.
\end{rem}
\begin{proof}
The assertion holds true if $\hte$ denotes Brown's entropy and $\alg$, hence $\blg$, are exact (\cite{NateBrown}). But for unital
completely positive maps on nuclear $C^*$-algebras the definition of the entropy given above coincides with that due to
N.\,Brown.
\end{proof}

We will also use the following version of the \emph{Kolmogorov-Sinai property}: if $(\Omega_i)_{i\in I}$ is a family of finite
subsets of $\alg$ such that $\bigcup_{i \in I, n \in \bn} \text{orb}^n(\Omega_i)$ is total in $\alg$ then $\hte (\gamma) =
\sup_{\varepsilon >0, i \in I} \lim \sup_{n \to \infty} \left( \tfrac{1}{n} \log \textrm{rcp}(\text{orb}^n(\Omega_i),\varepsilon
) \right)$. The proof is an easy modification of the one of Theorem 6.2.4 of \cite{book}. For further details and proofs and
various related topics we refer to \cite{book}. Note that $\Cun$ as a unital nuclear $C^*$-algebra falls into the class
considered above.

\subsection*{The Entropy Bound}
The main general result of this note is the following theorem. The proof is a generalisation of
that of Theorem 2.4 of \cite{Hous} (see also \cite{Boca}), now using Lemmas \ref{new} and
\ref{new2} instead of Lemma 2.2 of that paper. For the convenience of
the reader we give a terse reproduction of the proof with the necessary changes.

\begin{tw} \label{upest}
Let $k \in \bn$ and $u \in F_{k}$ be a unitary. Then \[ \hte(\uend)  \leq (k-1) \log N.\] In
particular if $\sigma$ is a permutation of $\Ind_k$, then this estimate holds true for the
corresponding permutation endomorphism $\edom$.
\end{tw}

\begin{proof}
Put for each $n \in \bn$
\[ \omega_l = \bigcup_{p,q=1}^l A_{p,q}.\]
Fix $l\in \bn$, $\delta >0$. As $\Cun$  is nuclear, there exists a triple $(\phi_0,\psi_0, M_{C_l}) \in CPA (\Cun,\omega_l,
\frac{1}{4N^{l}}\delta)$. Fix further $n\in \bn$ and let
\[ \omega_l^{(n)} = \bigcup_{p=0}^n \uend^p(\omega_l).\]
Put $m=n(k-1)+l$.
Nuclearity of $\Cun$ implies that there exists $d \in \bn$ and ucp maps $\gamma:\can_m(\Cun)\to M_d$ and $\eta:M_d\to \Cun$ such
that for all $a\in  \omega_l^{(n)}$,
\[
\| \eta \circ \gamma (a) - \can_m^{-1} (a) \| < \frac{\delta}{2}.
\]
Let $\mu:M_{N^m} \ot \Cun \to M_d$ be a ucp extension of $\gamma$.  Consider the following diagram:

\vspace*{0.4 cm}
\begin{picture}(350,100)
\put(0,80){$\Cun$} \put(80,80){$\can_m(\Cun)$} \put(92,50){$M_{N^m} \otimes \Cun$}
\put(325,80){$\Cun$} \put(205,80){$\can_m(\Cun)$}

\put(193,50){$M_{N^m} \otimes \Cun$}

\put(262,64){$M_d$}

\put(148,0){$M_{N^m} \otimes M_{C_l}$}

\put(39,88){$\can_m$} \put(20,83){\vector(1,0){50}}

\put(96,73){\vector(1,-1){14}}

\put(129,32){$\textrm{id} \ot \psi_0$} \put(116,42){\vector(1,-1){27}}

\put(176,32){$\textrm{id} \ot \phi_0$} \put(192,15){\vector(1,1){27}}

\put(224,65){\line(0,1){5}} \put(216,65){$\cap$}

\put(245,78){\vector(2,-1){14}} \put(245,68){$\gamma$}

\put(280,69){\vector(3,1){30}} \put(292,66){$\eta$}

\put(271,89){$\can_m^{-1}$} \put(247,84){\vector(1,0){62}}

\put(243,55){\vector(2,1){16}} \put(249,51){$\mu$}

\put(17,71){\vector(2,-1){123}} \put(65,33){$\psi$}

\put(205,11){\vector(2,1){113}} \put(265,31){$\phi$}

\end{picture}

\vspace{0.4cm}

\noindent
Consider now any $X \in \omega_l$ and let $p \in \bn$, $p\leq n$. Then
\begin{eqnarray*}
\lefteqn{\|\phi \circ \psi  (\uend^p(X)) - \uend^p(X) \|}  \hspace*{1.5cm}
 \\
&=&\|\eta \circ \mu \circ (\textrm{id} \ot \phi_0 \circ \psi_0)\circ \can_m  (\uend^p(X))
- (\can_m^{-1} \circ \can_m)(\uend^p(X)) \| \\
&\leq &\|  (\textrm{id} \ot \phi_0 \circ \psi_0) \circ \can_m  (\uend^p(X)) - \can_m(\uend^p(X)) \| +\frac{\delta}{2}.
\end{eqnarray*}
Lemma \ref{new2} implies that $\uend^p(X) \in F_{q,r}$, where $q,r \leq l+ (k-1)p \leq m$. We can assume that for example $q
>r$. Then Lemma \ref{new} implies that
\begin{eqnarray*}
\lefteqn{\|\phi \circ \psi  (\uend^p(X)) - \uend^p(X) \| } \hspace*{1.5cm}\\
&< & \Big\| \sum_{J \in \Ind_{q-r}} T_J \ot \big( (\phi_0 \circ \psi_0) ( s_J) - s_J \big) \Big\|+
 \frac{\delta}{2} \\
& < & 2 N^{m} \frac{\delta}{4N^m}+ \frac{\delta}{2} = \delta,
\end{eqnarray*}
and we proved that
\begin{equation}
\label{approx} (\phi,\psi,M_{N^m} \ot M_{C_l}) \in CPA (\Cun, \omega_l^{(n)}, \delta).
\end{equation}
This shows that $\textrm{rcp}(\omega_l^{(n)}, \delta)\leq C_l N^m$, \[\log \textrm{rcp}(\omega_l^{(n)}, \delta) \leq C_l + m \log
N= C_l + ((k-1)n + l)\log N\]  and finally
\[\limsup_{n\to\infty}
\left(\frac{1}{n} \log \textrm{rcp}(\omega_l^{(n)}, \delta)\right) \leq (k-1) \log N.\]
 The Kolmogorov-Sinai property for
noncommutative entropy  ends the proof.
\end{proof}

The proof above remains valid for any unital completely positive map on $\Cun$ which satisfies the conclusions of Lemma
\ref{new2}, and again can be suitably adapted to the context of (higher-rank) graph algebras $O_{\Lambda}$. As we are not aware
of any interesting and natural examples of such ucp maps for $O_{\Lambda}$ (apart from the canonical shifts in various directions
analysed in \cite{Hous}), we decided to present the result in the context of specific endomorphisms of $\Cun$.

If the endomorphism $\uend$ leaves the canonical masa $\nmasa$ invariant, it is also possible to obtain, exactly as in
\cite{Hous}, estimates for the noncommutative pressure (\cite{Kerr}) of any selfadjoint element of $\nmasa$. Note that in this
case however the estimate will not necessarily be optimal nor will it have to coincide with the pressure of the corresponding
element computed in $\nmasa$ viewed as the commutative subalgebra. This can be deduced from results in the next section.

It is easy to see that it is both possible to have $\hte (\uend) = 0$ (for the identity
endomorphism) and $\hte (\uend) = (k-1) \log N$ (for the endomorphism given by $(k-1)$-th iterate
of the canonical shift). Note that as $\uend$ need not leave the canonical masa in $\Cun$
invariant, we cannot always use the classical topological entropy to obtain the estimates from
below (as was done in \cite{Boca} and \cite{Hous}).

Let us stress that the examples below will show that even if $\uend$ leaves $\nmasa$ invariant it is not necessarily true that
$\hte(\uend) = \hte(\uend|_{\nmasa})$ although we will exhibit in each of the examples a special product masa ${\mathcal
C}=\ot_{i=1}^{\infty} {\mathcal C}_i$, where $ {\mathcal C}_i={\mathcal C}_j$ for all $i,j \in \bn$ such that $\hte(\uend) =
\hte(\uend|_{\mathcal C})$. Notice that for the canonical shift endomorphism  $\hte(\theta) =  \hte(\theta|_{\mathcal C})$ for
each special product masa $\mathcal C$.

\section{Entropy Values for certain permutative Endomorphisms of $\twoCun$}

This section is devoted to computing the entropy of all permutative endomorphisms $\edom:\twoCun \to \twoCun$, where $\sigma$ is a
permutation of $\Ind_2$. These 24 endomorphisms were listed and classified up to unitary equivalence
in \cite{Kawa}.
Note first that our entropy estimate in Theorem \ref{upest} implies that $\hte(\edom) \leq \log 2$.
Below we will compute the actual value of $\hte(\edom)$ (and of $\hte(\edom|_{\twoUHF})$, $\hte(\edom|_{\twomasa})$).
It turned out that in all cases either $\edom$ restricted to the standard masa or a suitable product masa gives
entropy $\log 2$ hence  $\hte(\edom) = \log 2$ or we could show directly that the entropy of $\edom$ is 0, so that
our methods are essentially commutative and do not use Voiculescu's definition.

The
notation will coincide with that of \cite{Kawa}; we identify $\Ind_2=\{(1,1),(1,2),(2,1),(2,2)\}$
with $\{1,2,3,4\}$. The subscripts in the symbols denoting permutations (e.g.\
$\sigma_{(12)(34)}$) will then correspond to the cycle decomposition of the permutation.
Thus for instance $u_{(1,2)}$ is given by:
\begin{eqnarray*}
u_{(1,2)}
&=&
s_1s_1(s_1s_2)^*+ s_1s_2(s_1s_1)^* + s_2s_1(s_2s_1)^*+s_2s_2(s_2s_2)^* \\
&=&
s_1s_1s_2^*s_1^* + s_1s_2s_1^*s_1^* + s_2s_1s_1^*s_2^*+s_2s_2s_2^*s_2^*.
\end{eqnarray*}
Therefore $\rho_{(1,2)}$ acts on the generators $s_1,s_2$ as follows
$$
\rho_{(1,2)}(s_1) = u_{(1,2)}s_1 = s_1s_1s_2^* + s_1s_2s_1^*
$$
and
$$
\rho_{(1,2)}(s_2) = u_{(1,2)}s_2 = s_2s_1s_1^* + s_2 s_2 s_2^* =s_2.
$$
The following table, modelled on that of \cite{Kawa}, summarises  the results of the entropy
computations. (In this table $s_{ij,k}$ denotes $s_is_js_k^*$.)

Notice that all automorphisms in this table have entropy 0 (these are $\rho_{id}$, $\rho_{(12)(34)}$, $ \rho_{(13)(24)}$,
$\rho_{(14)(23)}$). It is possible that permutative endomorphisms always have entropy 0.

Conti and Szyma\'nski (\cite{cs2}) have recently extended the table below by determining the indices of the
endomorphisms which shows that entropy and index seem to be related, although several combinations may occur.

{\footnotesize
\[
\begin{array}{ccccc}
\multicolumn{5}{c}{\mbox{Table 1. Entropy of the `rank 2' permutative endomorphisms of $\twoCun$.}}\\
\\
\hline \rho_{\sigma}& \rho_{\sigma}(s_{1})& \rho_{\sigma}(s_{2})&  \hte (\edom) & \hte(\edom|_{\twomasa})\\
 \hline \rho_{id} &
s_{1} & s_{2}&
0&0\\
\rho_{12}   & s_{12,1}+s_{11,2}& s_{2}&
 \log 2& 0\\
\rho_{13}   &s_{21,1}+s_{12,2}& s_{11,1}+s_{22,2}
& \log 2&\log 2\\
\rho_{14}   & s_{22,1}+s_{12,2}& s_{21,1}+s_{11,2}&
\log 2&\log 2\\
\rho_{23}   &s_{11,1}+s_{21,2}& s_{12,1}+s_{22,2}&
\log 2&\log 2\\
\rho_{24}   &s_{11,1}+s_{22,2}&s_{21,1}+s_{12,2}&
\log 2&\log 2\\
\rho_{34}   &s_{1}& s_{22,1}+s_{21,2}&
\log 2&0\\
\rho_{123} &s_{12,1}+s_{21,2}& s_{11,1}+s_{22,2}&
\log 2&\log 2\\
\rho_{132} &s_{21,1}+s_{11,2}& s_{12,1}+s_{22,2}&
\log 2&\log 2\\
\rho_{124} &s_{12,1}+s_{22,2}& s_{21,1}+s_{11,2}& \log 2&\log 2
\\
 \rho_{142} &s_{22,1}+s_{11,2}& s_{21,1}+s_{12,2}&
\log 2&\log 2\\
\rho_{134} &s_{21,1}+s_{12,2}& s_{22,1}+s_{11,2}& \log 2&\log 2
\\ \rho_{143} &s_{22,1}+s_{12,2}& s_{11,1}+s_{21,2}&
\log 2&\log 2\\
\rho_{234} & s_{11,1}+s_{21,2}& s_{22,1}+s_{12,2}&
\log 2&\log 2\\
\rho_{243} &s_{11,1}+s_{22,2}& s_{12,1}+s_{21,2}&
\log 2&\log 2\\
\rho_{1234} &s_{12,1}+s_{21,2}& s_{22,1}+s_{11,2}&
\log 2&\log 2 \\
\rho_{1243} &s_{12,1}+s_{22,2}& s_{11,1}+s_{21,2}&
\log 2&\log 2\\
\rho_{1324} &s_{2}& s_{12,1}+s_{11,2}&
\log 2&0\\
\rho_{1342} &s_{21,1}+s_{11,2}& s_{22,1}+s_{12,2}&
\log 2&\log 2\\
\rho_{1423} &s_{22,1}+s_{21,2}& s_{1}&
\log 2&0\\
\rho_{1432} &s_{22,1}+s_{11,2}& s_{12,1}+s_{21,2}&
\log 2&\log 2\\
\rho_{(12)(34)} &s_{12,1}+s_{11,2}& s_{22,1}+s_{21,2}&
0&0\\
\rho_{(13)(24)} &s_{2}& s_{1}&
0&0\\
\rho_{(14)(23)} &s_{22,1}+s_{21,2}& s_{12,1}+s_{11,2}&
0&0\\
\\
%
\end{array}
\]
}
By Lemma \ref{inv} each $\edom$ leaves the canonical masa invariant and we always have a
`commutative model' for our endomorphism.
To be more precise, the algebra $\twomasa$ is $*-$isomorphic to the algebra $C(\mathfrak{C})$
of continuous functions on the infinite product space
$\mathfrak{C}=\{(w_k)_{k=1}^{\infty}:w_k\in \{1,2\}\}$ (equipped with the usual metric
making it a $0$-dimensional compact space). The standard isomorphism is given by the
linear extension of the map $s_I s_I^* \mapsto \chi_{Z_I}$, where $Z_{I}$ denotes the
set of those sequences in $\mathfrak{C}$ which begin with the finite sequence  $I$.
Each of the endomorphisms $\edom$ restricted to $\twomasa$ corresponds therefore to
a continuous map $T_{\sigma}:\mathfrak{C} \to \mathfrak{C}$ via  $\edom (x) = x \circ T_{\sigma}$,
where we regard $x \in \twomasa$ as a continuous function on $\mathfrak{C}$.
Whilst it is well-known and easy to see that for the shift endomorphism
$T$ is simply the left shift on the sequence $(w_k)_{k=1}^{\infty}$ we have in general
only somewhat indirect information on how exactly $T_{\sigma}$ looks like for a given $\sigma$.
We note the following useful observations.
\begin{lem} \label{observe}
(i) Suppose that
$$
\edom(s_I s_I^*)= s_{I_1} s_{I_1}^* + \cdots + s_{I_m} s_{I_m}^*
$$
for a given $I \in \Ind$. Then  $s_{I_1} s_{I_1}^*, \cdots ,s_{I_m} s_{I_m}^*$ are pairwise
orthogonal and $T_{\sigma}$ restricts to a bijection
$$T_{\sigma}:\bigcup_{j=1}^m Z_{I_j}
\to Z_I.
$$
(ii) Suppose that for a given $k \in \bn$ and $i \in \{1,2\}$
\[ \edom(\sum_{J \in \Ind_k, j_k = i} s_J s_J^*) = \sum_{l=1}^n s_{I_l} s_{I_l}^*.\]
Then
\[(T_{\sigma}(w))_k = i \textrm{ if and only if } w \in \bigcup_{l=1}^n Z_{I_l}.\]
\end{lem}
\noindent
We will often also make
use of the following.
\begin{lem} \label{ent}
Suppose that $T:\mathfrak{C} \to \mathfrak{C}$ is a continuous map for which the following implication holds true:
if $k\in \bn$, $v,w \in \mathfrak{C}$ and $w|_{k+1]} \neq v|_{k+1]}$ then either
$(Tw)|_{k]} \neq (Tv)|_{k]}$ or $w|_{k]} \neq v|_{k]}$.
Then $h_{\textrm{top}}(T)
\geq \log 2$.
\end{lem}
\begin{proof}
Follow the usual argument using $(n,\epsilon)$-separated subsets of $\mathfrak{C}$ (see \cite{Walt}).
\end{proof}
\noindent
The rest of the paper is devoted to showing how to obtain the entropy values listed in Table 1.

\subsection*{Identity map} Entropy $0$.

\subsection*{Shift endomorphism}
Arises from $\sigma_{23}$. The entropy is equal to $\log 2$ (see \cite{Boca}). The shift $\theta$
leaves the canonical masa invariant and $\hte(\theta) = \hte(\theta|_{\twoUHF}) =
\hte(\theta|_{\twomasa})$.

\subsection*{Flip transformation}
Arises from $\sigma_{(13)(24)}$, acts as $\edom(s_1) = s_2$, $\edom(s_2) = s_1$. The entropy is
$0$; this follows immediately from general results of \cite{ds}, but can be also easily deduced
directly.

\subsection*{The transformation induced by $\sigma_{12}$}
Denote it simply by $\psi$. It is defined by
\[ \psi(s_1) = s_1 s_2 s_1^* + s_1 s_1 s_2^*, \;\;\; \psi(s_2) = s_2.\]
We have ($n,m \in \bn$)
\[ \psi(s_1^n) = s_1 s_2^n s_1^* + s_1 s_2^{n-1}s_1 s_2^*,\;\; \psi(s_2^m) = s_2^m\]
(the first formula can be easily shown inductively, and the second  is obvious). This leads to
\[ \psi(s_1^{i_1} s_2^{j_1} \cdots s_2^{j_{k-1}} s_1^{i_k}) = s_1 s_2^{i_1-1} s_1 s_2^{j_1-1}
\cdots s_1 s_2^{j_{k-1}-1} s_1 (s_2^{i_k-1} s_1 s_2^* + s_2^{i_k} s_1^*),\]
\[ \psi(s_1^{i_1} s_2^{j_1} \cdots s_2^{j_{k-1}} s_1^{i_k}s_2^{j_k} ) = s_1 s_2^{i_1-1} s_1 s_2^{j_1-1}
\cdots s_1 s_2^{j_{k-1}-1} s_1 s_2^{i_k-1} s_1 s_2^{j_k -1}\] ($k \in \bn$, $i_1,\ldots,i_k,j_1,\ldots,j_k \in \bn$).
Further if $s_{\nu} =s_1^{i_1} s_2^{j_1} \cdots s_2^{j_{k-1}} s_1^{i_k}$ then
\[ \psi(s_{\nu} s_{\nu}^*) = s_{\nu_1}s_{\nu_1}^* + s_{\nu_2}s_{\nu_2}^*,\]
where \[s_{\nu_1} = s_1 s_2^{i_1-1} s_1 s_2^{j_1-1} \cdots s_1 s_2^{j_{k-1}-1} s_1 s_2^{i_k-1} s_1,\]
\[s_{\nu_2} = s_1 s_2^{i_1-1} s_1
s_2^{j_1-1} \cdots s_1 s_2^{j_{k-1}-1} s_1 s_2^{i_k}.\] This shows immediately that
\[ \psi(s_{\nu} s_{\nu}^*) = s_{\wt{\nu}} s_{\wt{\nu}} ^*, \]
where
\[ s_{\wt{\nu}} =  s_1 s_2^{i_1-1} s_1 s_2^{j_1-1}
\cdots s_1 s_2^{j_{k-1}-1} s_1 s_2^{i_k-1}. \] If we have an index ending with $2$, so that $s_\mu
=s_1^{i_1} s_2^{j_1} \cdots s_2^{j_{k-1}} s_1^{i_k}s_2^{j_k}$, then
\[ \psi(s_{\mu} s_{\mu}^*) = s_{\wt{\mu}} s_{\wt{\mu}} ^*, \]
where
\[ s_{\wt{\mu}} = s_1 s_2^{i_1-1}
s_1 s_2^{j_1-1} \cdots s_1 s_2^{j_{k-1}-1} s_1 s_2^{i_k-1} s_1 s_2^{j_k -1}.\]
Note that each occurrence of $s_1$ in $\psi(s_{\mu}s_{\mu}^*)$ is caused by a `change' in the
sequence represented by $\mu$ (if we assume that all sequences $\mu$ have $s_2$ as the $0$-th
element). This observation implies that any sequence $\mu$ ending in $2$ gives an output sequence
with an even number of $1$'s (even number of changes), so that $\psi|_{\twomasa}$ is induced by the
transformation of $\mathfrak{C}$ given by
\[ (T_{\psi}(w))_k = \begin{cases} 1 & \textrm{if }  \sharp \{j \leq k: w_j=1\} \textrm{ is odd,} \\
                           2 & \textrm{if }  \sharp\{j \leq k: w_j=1\} \textrm{ is even.} \end{cases}\]
It is easy to see that $h_{\textrm{top}}(T_{\psi}) = 0$, so also $\hte(\psi|_{\twomasa})=0$.

We will see that there is another masa in $\Cun$ which is left invariant by $\psi$ and such that the corresponding restriction
has entropy $\log 2$. Let $X= s_1 s_2^* + s_2 s_1^*$. Then
\[ \psi(X) = s_1 s_2 s_1^* s_2^* + s_1 s_1 s_2^* s_2^* + s_2 s_1 s_2^* s_1^* + s_2 s_2 s_1^* s_1^* = s_1 X s_2 ^* + s_2 Xs_1^*.\]
Moreover if $\theta:\twoCun \to \twoCun$ denotes the canonical shift endomorphism then
\[ \psi (\theta (X)) = \psi(s_1 X s_1^* + s_2 X s_2^*) = \psi(s_1) (s_1 X s_2^* + s_2 X s_1^*) \psi(s_1^*) + s_2 \psi(X) s_2^*
= \theta (\psi (X)).\] We will now show that for each $k \in \bn$ \begin{equation} \label{commuteX} \theta^k (\psi (X)) =
\psi(\theta^k(X))\end{equation}
 Suppose we have shown for some $n \in \bn$
that $\theta^n (\psi (X)) = \psi(\theta^n(X))$. Then
\begin{align*}\psi(\theta^{n+1} (X)) &= \sum_{\mu \in \Ind_{n+1}} \psi (s_{\mu} X s_{\mu}^*) \\& = \sum_{\nu \in \Ind_{n}} \left(\psi (s_1)
\psi(s_{\nu} X s_{\nu}^*) \psi(s_1^*) + s_2 \psi(s_{\nu} X s_{\nu}^*) s_2^*  \right) \\&= \psi (s_1) \sum_{\nu \in \Ind_n}
s_{\nu} \psi (X) s_{\nu}^* \psi(s_1)^* + s_2 \sum_{\nu \in \Ind_n} s_{\nu} \psi (X) s_{\nu}^* s_2^*  \\&= \sum_{\mu \in
\Ind_{n+1}} s_{\mu} \psi(X) s_{\mu}^* = \theta^{n+1}(\psi(X)).
\end{align*}
The second last equality follows if we notice that $\psi(s_1) s_{\nu} = s_1 s_{\wt{\nu}}$, where $\wt{\nu}$ equals $\nu$ but
with the first letter `switched'.

The formula \eqref{commuteX} will become useful when we view the UHF algebra $\twoUHF$ as the
tensor product $\bigotimes_{i=1}^{\infty} M_2^{(i)}$. Define first
\[ E= \frac{1}{2} (I + X), \;\;\; F = \frac{1}{2} (I-X).\]
It is clear that $E$ and $F$ are minimal projections in the first matrix factor of the UHF algebra. Thus the algebra generated by
$\{\theta^n(E), \theta^n(F): n \in \bn_0\}$ is a masa, further denoted by $\mathcal{C}_{E,F}$.  Because of
\eqref{commuteX} we immediately see that also
\[\theta^k (\psi (E)) = \psi(\theta^k(E)), \;\;\; \theta^k (\psi (F)) =
\psi(\theta^k(F))\] for all $k \in \bn$. As in the tensor picture $\psi(X) = X \ot X$, it is easy to see that
\[ \psi(E) = E \ot E + F \ot F, \;\;\; \psi (F) = F \ot F + E \ot E.\]
In conjunction with the previous statement this implies that $\psi$ leaves $\mathcal{C}_{E,F}$
invariant. The algebra $\mathcal{C}_{E,F}$ is isomorphic to $C(\mathfrak{C})$. The isomorphism may
be given for example by identifying $E$ with $\chi_{Z_1}$ and $F$ with $\chi_{Z_2}$ so that for
example $E \ot F \ot E \ot E$ is mapped to $ \chi_{Z_{1211}}$. It is easy to show that $T_{E,F}$,
the induced continuous map on $\mathfrak{C}$, is given by the formula:
\[(T_{E,F}(w))_k = \begin{cases} 1 & \textrm{if }  w_k=w_{k+1}, \\
                           2 & \textrm{if }  w_k \neq w_{k+1}. \end{cases}\]
By lemma \ref{ent} $\hte
(\psi_{\mathcal{C}_{E,F}}) = h_{\textrm{top}}(T_{E,F}) = \log 2$. Together with Theorem \ref{upest}
this implies that
\[ \hte(\psi) = \hte(\psi|_{\twoUHF}) = \log 2.\]


\subsection*{The transformation induced by $\sigma_{1324}$}
Let $\psi$ denote again the endomorphism induced by $\sigma_{12}$ and let $\psi'$ denote the one induced by $\sigma_{1324}$. Then
\[\psi' (s_1) = \psi(s_2), \;\; \psi'(s_2) = \psi'(s_1).\]
Note that this implies in particular that on the masa $\mathcal{C}_{E,F}$ introduced earlier the endomorphisms $\psi'$ and $\psi$
coincide. Indeed, $\psi(E) = \psi'(E)$ and also
\begin{align*} \psi' (\theta^n(E)) &= \psi' (\sum_{\mu \in \Ind_n} s_{\mu} E
s_{\mu}^*) = \sum_{\mu \in \Ind_n} \psi'(s_{\mu}) \psi'(E) \psi'(s_{\mu}^*) \\&= \sum_{\mu \in \Ind_n} \psi(s_{\mu}) \psi(E)
\psi(s_{\mu}^*) = \psi(\theta^n(E)).\end{align*}
 Thus $\hte(\psi') \geq \hte(\psi'|_{\mathcal{C}_{E,F}}) = \hte(\psi|_{\mathcal{C}_{E,F}})= \log 2$ and we
obtain
\[ \hte(\psi) = \hte(\psi|_{\twoUHF}) = \log 2.\]
Note that as $\psi'(\sum_{J \in \Ind_k, j_k = 1} s_J s_J^*)= \psi(\sum_{J \in \Ind_k, j_k = 2} s_J
s_J^*)$, the restriction of $\psi'$ to $\twomasa$ is isomorphic to the map given by
\[ (T_{\psi'}(w))_k = \begin{cases} 2 & \textrm{if }  \sharp \{j \leq k: w_j=1\} \textrm{ is odd} \\
                           1 & \textrm{if }  \sharp\{j \leq k: w_j=1\} \textrm{ is even} \end{cases}\]
and thus $\hte(\psi'|_{\twomasa})=0$.


\subsection*{The transformation induced by $\sigma_{(14)(23)}$}
It is shown in \cite{Kawa} that this endomorphism is given by the conjugation with the unitary $s_1
s_2^* + s_2 s_1^*$. Thus it leaves each of the subspaces $F_{p,l}$ ($p,l \in \bn$) invariant and
one can easily deduce using the Kolmogorov-Sinai property that its entropy is $0$.

\subsection*{The transformation induced by $\sigma_{(12)(34)}$}
This one is the composition of the flip automorphism and $\rho_{(14)(23)}$. As
they both leave $F_{p,l}$ invariant, the entropy is $0$.

\subsection*{The transformations induced by $\sigma_{14}, \sigma_{132},
\sigma_{124}, \sigma_{143}, \sigma_{234},$ $ \sigma_{1243}, \sigma_{1342}$}

Let $\sigma$ be one of the permutations from the above list. It is easy to show inductively that for any $k \in \bn$ and $J \in
\Ind_k$ we have
\begin{equation} \edom(s_{J}) = s_{J_1} s_1^* + s_{J_2} s_2^*,\label{genform}\end{equation}
where $J_1,J_2$ are certain multi-indices  in $\Ind_k$. This implies, as we will show below, that in some special cases the
formula for the map $T_{\edom}:\mathfrak{C} \to \mathfrak{C}$ induced by the restriction of $\psi$ to $\twomasa$ is determined already by the value of
$\psi(s_1 s_1^*)$. Let us formulate it as a lemma:

\begin{lem} \label{2cases}
Suppose that the endomorphism $\edom: \twoCun \to \twoCun$ satisfies the condition \eqref{genform}. If
$\edom(s_1 s_1^*) = s_1 s_2 s_2^* s_1^* + s_2 s_2 s_2^* s_2^*$ then
\[(T_{\edom}(w))_k = \begin{cases} 1 & \textrm{if }  w_{k+1}=2, \\
                           2 & \textrm{if }  w_{k+1}=1. \end{cases}\]
If $\edom(s_1 s_1^*) = s_1 s_1 s_1^* s_1^* + s_2 s_1 s_1^* s_2^*$ then
\[(T_{\edom}(w))_k = \begin{cases} 1 & \textrm{if }  w_{k+1}=1, \\
                           2 & \textrm{if }   w_{k+1}=2. \end{cases}\]
\end{lem}
\begin{proof}
It is enough to consider the first case, the second follows in an analogous way. Suppose that $J \in \Ind$ and $J_1, J_2$ are as
in the formula \eqref{genform}.  Then
\begin{align*}\edom(s_J s_1 s_1^*s_J^*) &= (s_{J_1} s_1^* + s_{J_2} s_2^*)(s_1 s_2 s_2^*
s_1^* + s_2 s_2 s_2^* s_2^*) (s_{1} s_{J_1}^* + s_2 s_{J_2}^*) \\&= s_{J_1} s_2 s_2^* s_{J_1}^* + s_{J_2} s_2 s_2^*
s_{J_2}^*.\end{align*} This implies that
\[\edom(\sum_{J \in \Ind_k, j_k = 1} s_J s_J^*) = \sum_{I \in \Ind_{k+1},i_{k+1}= 2} s_{I} s_{I}^*\]
and we can finish the proof using \ref{observe}.(ii).
\end{proof}

The analysis of the values at $s_1 s_1^*$ together with Theorem \ref{upest}, Lemma \ref{2cases} and
\ref{ent} show that for any $\sigma$ from the list $
\sigma_{14}, \sigma_{132}, \sigma_{124}, \sigma_{143}, \sigma_{234},$ $ \sigma_{1243},
\sigma_{1342}$ we have
\[ \hte(\edom)= \hte(\edom|_{\twomasa}) = \log 2.\]

\subsection*{The transformation induced by $\sigma_{13}$}

Let $\sigma= \sigma_{13}$. It is easy to see that actually in this case the formula \eqref{genform}
can be made more precise so that we obtain for any $k \in \bn$ and $J \in \Ind_k$
\[ \edom(s_{J}) = s_{J_1} s_{1} s_1^* + s_{J_2}s_2 s_2^*,\] where now $J_1,J_2 \in
\Ind_{k-1}$. Moreover we have
\[ \edom(s_1 s_1^*) = s_1 s_2 s_2^* s_1^* + s_2 s_1 s_1^* s_2^*,\]
so that
\begin{align*}  \edom(s_{J} s_{1} s_1^* s_J^*) &= (s_{J_1} s_{1} s_1^* + s_{J_2}s_2 s_2^*) (s_1 s_2 s_2^* s_1^* + s_2 s_1
s_1^* s_2^*) (s_{J_1} s_{1} s_1^* + s_{J_2}s_2 s_2^*)^* \\&= s_{J_1} s_1 s_2 s_2^* s_1^* s_{J_1}^* + s_{J_2} s_2 s_1 s_1^* s_2^*
s_{J_2}^*.\end{align*}
 This implies easily that $\edom$ restricted to $\twomasa$ is induced by the map
\[(T_{\edom}(w))_k = \begin{cases} 1 & \textrm{if }  w_k\neq w_{k+1}, \\
                           2 & \textrm{if }  w_k = w_{k+1}. \end{cases}\]
As in the last subsection we obtain
\[ \hte(\edom)= \hte(\edom|_{\twomasa}) = \log 2.\]


\subsection*{The transformation induced by $\sigma_{1432}$}
The endomorphism is the composition of the inner automorphism $\rho_{(1243)}$ with $\rho_{(13)}$. This implies that
$\edom$ restricted to $\twomasa$ is induced by the map
\[(T_{\edom}(w))_k = \begin{cases} 1 & \textrm{if } k \geq 2 \textrm{ and } w_k\neq w_{k+1} \textrm{ or } k=1 \textrm{ and } w_1=w_2 , \\
                           2 & \textrm{if } k \geq 2 \textrm{ and } w_k= w_{k+1} \textrm{ or } k=1 \textrm{ and } w_1\neq w_2, \end{cases}\]
and we obtain
\[ \hte(\edom)= \hte(\edom|_{\twomasa}) = \log 2.\]


\subsection*{The transformation induced by $\sigma_{123}$}
The endomorphism is given by the formulas
\[ \edom(s_1) = s_1 s_2 s_1^* + s_2 s_1 s_2^*, \;\; \edom(s_2) = s_1 s_1 s_1^* + s_2 s_2 s_2^*,\]
so that
\[ \edom(s_1 s_1^*) = s_1 s_2 s_2^* s_1^* + s_2 s_1 s_1^*s_2^*,\]
\[ \edom(s_2 s_2^*) = s_1 s_1 s_1^* s_1^* + s_2 s_2 s_2^* s_2^*.\]
It is also easy to check that it has the property described in \eqref{genform}. Moreover
the following lemma  holds true.

\begin{lem} \label{st1}
Let $k \in \bn_0$, $J \in \Ind_k$. Then
\[ \edom(s_{J}s_1 s_1^* s_J^*) = s_{J_1} s_{J_1}^* + s_{J_2} s_{J_2}^*,\]
where $J_1, J_2$ are certain indices in $\Ind_{k+1}$ such that the number of constant segments of 1's and 2's in  $J_{1}$ and in  $J_{2}$ is  even.
\end{lem}
\begin{proof}
The statement will be proved by induction on $k$. The case $k=0$ follows from the explicit formulae before the lemma. Let
now $J \in \Ind_{k}$ and compute
\[\edom(s_1 s_J s_J^* s_1^*) = (s_1 s_2 s_1^* + s_2 s_1 s_2^*) (s_{J_1} s_{J_1}^* + s_{J_2} s_{J_2}^*) (s_1 s_2 s_1^* + s_2 s_1
s_2^*)^*.\] Suppose first that  $J_1 = 1 K$ for some $K \in \Ind_{k}$. Then
\[ (s_1 s_2 s_1^* + s_2 s_1 s_2^*) (s_{J_1} s_{J_1}^*) (s_1 s_2 s_1^* + s_2 s_1 s_2^*)^* = s_1 s_2 s_K s_K^* s_2^* s_1^*.\]
We want to count the constant segments in the multi-index $12K$. A moment of thought shows that it has either equally many
constant segments as $1K$ (if $K$ began with $2$) or two more (if $K$ began with $1$).  Similarly if $J_1 = 2K$ for some $K
\in \Ind_{k-1}$ we have
\[ (s_1 s_2 s_1^* + s_2 s_1 s_2^*) (s_{J_1} s_{J_1}^*) (s_1 s_2 s_1^* + s_2 s_1 s_2^*)^* = s_2 s_1 s_K s_K^* s_1^* s_2^*,\]
and again the multi-index $21K$ has either equally many constant segments as $2K$ (if $K$ originally began with $1$) or two more
(if $K$ originally began with $2$).

It remains to consider what happens when on the left $J$ is extended by $2$ instead of $1$:
\[ \edom(s_2 s_J s_J^* s_2^*) =(s_1 s_1 s_1^* + s_2 s_2 s_2^*) (s_{J_1} s_{J_1}^* + s_{J_2} s_{J_2}^*) (s_1 s_1 s_1^* + s_1 s_1
s_1^*)^*.\] Suppose first that $J_1 = 1 K$ for some $K \in \Ind_{k}$. Then
\[ (s_1 s_1 s_1^* + s_2 s_2 s_2^*) (s_{J_1} s_{J_1}^*) (s_1 s_1 s_1^* + s_2 s_2 s_2^*)^* = s_1 s_1 s_K s_K^* s_1^* s_1^*.\]
The multindex  $11K$ has obviously equally many constant segments as $1K$. An analogous argument suffices if $J_1 = 2 K$ for
some $K \in \Ind_{k-1}$ and the inductive proof is finished - the parity of the number of constant segments in the multi-indices
appearing in the $(k+1)$-th stage is the same as in those which appeared in the $k$-th stage.
\end{proof}

The lemma above implies that the map on $\mathfrak{C}$ induced by $\edom|_{\twomasa}$ is given by
the formula:
\[(T_{\edom}(w))_k = \begin{cases} 1 & \textrm{if the number of constant segments in }  w|_{k+1]} \textrm{ is  even,} \\
                           2 & \textrm{if the number of constant segments in }  w|_{k+1]} \textrm{ is  odd}. \end{cases}\]
This implies again that
\[ \hte(\edom)= \hte(\edom|_{\twomasa}) = \log 2.\]


\subsection*{The transformation induced by $\sigma_{142}$}
We will apply a method analogous to that used for $\sigma_{123}$. Here the endomorphism is given by the formulas
\[ \edom(s_1) = s_2 s_2 s_1^* + s_{1} s_1 s_2^*, \;\; \edom(s_2) = s_2 s_1 s_1^*+ s_1 s_2 s_2^*,\]
so that
\[ \edom(s_1 s_1^*) = s_2 s_2 s_2^* s_2^* + s_1 s_1 s_1^* s_1^*,\]
\[ \edom(s_2 s_2^*) = s_2 s_1 s_1^* s_2^* + s_1 s_2 s_2^* s_1^*.\]
 It is also easy to check that it has the property described in \eqref{genform}.

The next lemma is analogous to Lemma \ref{st1} and its proof is omitted.
\begin{lem}
Let $k \in \bn_0$, $J \in \Ind_k$. Then
\[ \edom(s_{J}s_1 s_1^* s_J^*) = s_{J_1} s_{J_1}^* + s_{J_2} s_{J_2}^*,\]
where $J_1, J_2$ are certain indices in $\Ind_{k+1}$ such that $k +$ the number of constant segments in  $J_{1}$ is  even and $k
+$ the number of constant segments in  $J_{2}$ is  even.
\end{lem}
%

It follows that
\[(T_{\edom}(w))_k = \begin{cases} 1 & \textrm{if } k + \textrm{the number of constant segments in }  w|_{k+1]} \textrm{ is  odd,} \\
                           2 & \textrm{if } k + \textrm{the number of constant segments in }  w|_{k+1]} \textrm{ is  even}. \end{cases}\]
Thus we once more obtain
\[ \hte(\edom)= \hte(\edom|_{\twomasa}) = \log 2.\]

\subsection*{The transformations induced by $\sigma_{34}, \sigma_{1423}, \sigma_{24}, \sigma_{1234}, \sigma_{243}$ and $\sigma_{134}$}
Arise from  $\sigma_{12}, \sigma_{1324}, \sigma_{13}, \sigma_{1432}, \sigma_{123}$ and $\sigma_{142}$ respectively, just by $1$ and
$2$ switching places.  The entropy values can  thus be read off from the earlier computations.

\begin{rem}
As in all the cases above the maximal value of the topological entropy is attained on a commutative subalgebra, the variational
principle has to hold for all $\edom$. Recall that this means that $\hte(\edom) = \sup_{\phi} h_{\phi} (\edom)$, where the
supremum is taken over all states on $\twoCun$ left invariant by $\edom$ and $h_{\phi}(\edom)$ denotes the dynamical state
entropy of Connes, Narnhofer and Thirring (see \cite{book}). It can be easily seen that in each case the supremum is realised by the state $\tau \circ \be$.
\end{rem}

\end{document}